\documentclass[11pt]{amsart} 
\usepackage{amsmath,amsfonts,amssymb}
\usepackage{amsthm}
\usepackage{xfrac}
\usepackage{hyperref}
\usepackage[all]{xy}

\usepackage{color}

\theoremstyle{plain}
\newtheorem{thm}{Theorem}
\newtheorem{lem}[thm]{Lemma}
\newtheorem{prop}[thm]{Proposition}
\newtheorem{prob}[thm]{Problem}
\newtheorem{cor}[thm]{Corollary}
\newtheorem*{claim}{Claim}

\newcommand{\fakeenv}{} 

\newenvironment{restate}[2]  
{
  \renewcommand{\fakeenv}{#2} 
  \theoremstyle{plain}
  \newtheorem*{\fakeenv}{#1~\ref{#2}} 
  \begin{\fakeenv}
}
{
  \end{\fakeenv}
}

\DeclareMathOperator{\Out}{Out}

\theoremstyle{definition}
\newtheorem{defn}[thm]{Definition}

\newtheorem{eg}[thm]{Example}
\newtheorem{rmk}[thm]{Remark}

\numberwithin{equation}{section}
\numberwithin{thm}{section}

\usepackage[utf8]{inputenc} 

\usepackage{geometry} 
\usepackage{graphicx} 

\title{\em Irreducible Nonsurjective Endomorphisms of $F_n$ are Hyperbolic}
\author{Jean Pierre Mutanguha}

\address{Department of Mathematical Sciences \\ University of Arkansas \\ Fayetteville, AR \newline \indent  \it Web address: \tt {\url{https://mutanguha.com}}}
\email{\href{mailto:jpmutang@uark.edu}{jpmutang@uark.edu}}

\date{} 

\begin{document}

\begin{abstract} Previously, Reynolds showed that any irreducible nonsurjective endomorphism can be represented by an irreducible immersion on a finite graph. We give a new proof of this and also show a partial converse holds when the immersion has connected Whitehead graphs with no cut vertices. The next result is a characterization of finitely generated subgroups of the free group that are invariant under an irreducible nonsurjective endomorphism. Consequently, irreducible nonsurjective endomorphisms are fully irreducible. The characterization and Reynolds' theorem imply that the mapping torus of an irreducible nonsurjective endomorphism is word-hyperbolic.
\end{abstract}
\maketitle

\section{Introduction}

Bestvina-Handel defined irreducible automorphisms of free groups in \cite{BH92} as the base case of their proof of the Scott conjecture. Irreducible automorphisms are the dynamical free group analogues of pseudo-Anosov diffeomorphisms of surfaces and, as such, the study of their dynamics has been an active area of research since their introduction. For instance, a leaf of the attracting lamination of a pseudo-Anosov diffeomorphism cannot have compact support in an infinite cover of the surface, and analogously, Bestvina-Feighn-Handel constructed a lamination for fully irreducible automorphisms and showed that a leaf of the lamination cannot be supported by a finitely generated infinite index subgroup \cite[Proposition~2.4]{BFH97}. Another analogy: pseudo-Anosov mapping classes act with north-south dynamics on the Thurston compactification of Teichm\"uller space \cite{Thu88} and Levitt-Lustig proved that fully irreducible automorphisms act with north-south dynamics on the compactification of  Culler-Vogtmann outer space \cite{LL03}. We would like to extend these results to irreducible endomorphisms.

We begin by defining the (semi)action of an injective endomorphism on outer space and the first result is that irreducible nonsurjective endomorphisms act with sink dynamics on outer space. There is a correspondence between having a fixed point of the action and being able to represent the endomorphism with a graph immersion. Under this correspondence, we prove:

\begin{restate}{Theorem}{part2} If $\phi:F\to F$ is an irreducible nonsurjective endomorphism, then it can be represented by an irreducible immersion with connected Whitehead graphs. Moreover, the immersion is unique amongst irreducible immersions.
\end{restate}

All the terms in the theorem will be defined in Section \ref{defs}. This is an unpublished result of Patrick Reynolds \cite[Corollary~5.5]{Rey} but the proof given here is original. We also note that Reynolds' proof assumes the endomorphism is fully irreducible but, as we shall see shortly, this is equivalent to being irreducible. Not much is currently known about the action; for instance,

\begin{prob}Does the action of an irreducible nonsurjective endomorphism on outer~space always have bounded image? See Example~\ref{sapir2}.
\end{prob}

Using the sink dynamics on outer space, we prove a partial converse of Theorem~\ref{part2}:

\begin{restate}{Theorem}{part3} If $\phi:F \to F$ is represented by an irreducible immersion whose Whitehead graphs are connected and have no cut vertices, then $\phi$ is nonsurjective and fully irreducible.
\end{restate}

Example~\ref{countereg} below shows that this criterion is not necessary, which raises the question:

\begin{prob}Is there an algorithmic characterization for fully irreducible nonsurjective endomorphisms?
\end{prob}

The next proposition essentially determines which finitely generated subgroups are invariant under irreducible endomorphisms; roughly speaking, these are the iterated images of the endomorphism, up to finite index.

\begin{restate}{Proposition}{invSbgrp}Suppose $\phi : F \to F$ is injective and represented by an irreducible train track with connected Whitehead graphs. If $H \le F$ is a finitely generated subgroup such that $\phi(H) \le H$ and $H$ contains a $\phi$-expanding conjugacy class, then \[ [\phi^k(F) : \phi^k(F) \cap H] < \infty \quad \text{for some } k \ge 0.\]
\end{restate}

This follows from the technical Proposition~\ref{supSbgrp}. We then get a characterization for fully irreducible endomorphisms that applies to all injective endomorphisms.

\begin{restate}{Theorem}{propKap}Let $\phi:F \to F$ be an injective endomorphism. Then $\phi$ is fully irreducible if and only if $\phi:F \to F$ has no periodic cyclic free factor, it is represented by an irreducible train track with connected Whitehead graphs, and its image $\phi(F)$ is not contained in a proper free factor.
\end{restate}

Dowdall-Kapovich-Leininger had previously shown the equivalence between the irreducible and fully irreducible properties for atoroidal automorphisms  \cite[Corollary~B.4]{DKL15}.
Theorem~\ref{part2} and Theorem~\ref{propKap} give us this equivalence for nonsurjective endomorphisms.

\begin{restate}{Corollary}{myequiv} If $\phi:F \to F$ is irreducible and nonsurjective, then it is fully irreducible.
\end{restate}

Another consequence of Proposition~\ref{invSbgrp} is the hyperbolicity of the mapping torus of an irreducible immersion with connected Whitehead graphs -- this was our original motivation.

\begin{restate}{Theorem}{hypthm} If $\phi:F \to F$ is represented by an irreducible immersion with connected Whitehead graphs, then $F*_\phi$ is word-hyperbolic. In particular, if $\phi:F \to F$ is nonsurjective and irreducible, then $F*_\phi$ is word-hyperbolic.
\end{restate}

By our previous work \cite{JPM}, the proof boils down to showing that irreducible nonsurjective endomorphisms cannot have periodic laminations. As a quick application, we prove the hyperbolicity of the {\it Sapir group}:

\begin{eg}\label{sapir}Let $F = F(a,b)$ be the free group on two generators and $\phi:F \to F$ be the nonsurjective endomorphism given by $\phi(a) = ab$ and $\phi(b) = ba$. The obvious map on the standard rose $f: R_2 \to R_2$ will be an immersion that induces $\phi$ on the fundamental group. It is easy to verify that $f$ is an irreducible immersion whose Whitehead graph is connected and has no cut vertices. By Theorem~\ref{part3}, $\phi$ is fully irreducible, and by Theorem \ref{hypthm}, the Sapir group $F*_\phi = \langle a, b, t ~|~ t^{-1}at = ab, t^{-1}bt = ba \rangle$ is word-hyperbolic.

In contrast, $\psi:F \to F$ given by $(a,b)\mapsto (aba, bab)$ is not word-hyperbolic: $\psi(ab) = (ab)^3$ and so $BS(1,3) \cong \langle ab, t \rangle \le F*_\psi$. As $H_1(F*_\psi) = \mathbb Z^2$ has rank $>1$, there are infinitely many isomorphisms $F*_\psi \cong F_n*_{\psi_n}$ \cite{BNS87}; all such endomorphisms $\psi_n:F_n \to F_n$ are nonsurjective, because $BS(1,3) \le F*_\psi$, and reducible by Theorem~\ref{hypthm}.
\end{eg}

In future work, we use Proposition~\ref{invSbgrp} to show that the irreducibility of an endomorphism is a group-invariant of its mapping torus, i.e., suppose $\phi:F_n \to F_n$ and $\psi:F_m \to F_m$ are endomorphisms whose images are not contained in proper free factors and $F_n*_\phi \cong F_m *_\psi$, then $\phi$ is irreducible if and only if $\psi$ is irreducible~\cite{JPM2}. This answers Question 1.4 posed by Dowdall-Kapovich-Leininger in \cite{DKL17}; see also \cite[Theorem~C]{DKL15}.

{~}

\noindent \textbf{Outline:} Section~\ref{defs} contains the standard definitions and preliminary results that will be used throughout the paper. We define outer space and the action of injective endomorphisms on it in Section~\ref{cvf}. Section~\ref{immersions} contains the proof of Theorem~\ref{part2} and a partial converse Theorem~\ref{part3}. In Section~\ref{iwip}, we prove the technical result classifying subgroups that support a leaf of a lamination, Proposition~\ref{supSbgrp}. While the two previous sections are independent, we combine their main results in Section~\ref{hyperbolic} to prove Theorem~\ref{hypthm}.

{~}

\noindent \textbf{Acknowledgments:} I would like to thank Ilya Kapovich for getting me started on the proof of Theorem \ref{part2}. I also thank Derrick Wigglesworth and my advisor Matt Clay for the countless conversations on this material. The clarity of my exposition benefited from the referee's comments.

\section{Definitions and Preliminary Results}\label{defs}

In this paper, $F$ is a finitely generated free group on at least two generators.

\begin{defn} An endomorphism $\phi: F \to F$ is {\bf reducible} if there exists a free factorization $A_1 * \cdots * A_k * B $ of $F$, where $B$ is nontrivial if $k = 1$, and a sequence of elements, $(g_i)_{i=1}^k$,  in $F$ such that $\phi(A_i) \le g_i A_{i+1} g_i^{-1}$ where the indices are considered$\mod k$. An endomorphism $\phi$ is {\bf irreducible} if it is not reducible and it is {\bf fully irreducible} if all its iterates are irreducible; equivalently, $\phi$ is fully irreducible if $\phi$ has no invariant proper free factor, i.e., there does not exist a proper free factor $A \le F$, an element $g \in F$, and an integer $n \ge 1$ such that $\phi^n(A) \le gAg^{-1}$.
\end{defn}

An application of Stallings folds implies the endomorphisms studied in this paper will be injective even though it may not be explicitly stated again.
\begin{lem}\label{inj} If $\phi:F \to F$ is irreducible, then it is injective.\end{lem}
\begin{proof} Suppose $\phi$ is not injective. Stallings \cite{St83} showed that any endomorphism of $F$ factors through folds whose nontrivial kernels corresponding to proper subgraphs, hence are normal closures of proper free factors. Thus $\ker \phi$ contains an invariant proper free factor, which contradicts the irreducibility assumption.
\end{proof}

\begin{defn} Fix an isomorphism $F \cong \pi_1(\Gamma)$ for some connected finite core graph $\Gamma$, i.e., a finite 1-dimensional CW-complex with no univalent vertices. For any subgroup $H \le F$, the {\bf Stallings subgroup graph} $S(H)$ is the core of the cover of $\Gamma$ corresponding to $H$, i.e., it is the smallest subgraph of the cover containing all immersed loops of the cover. When $H$ is nontrivial and finitely generated, $S(H)$ is a finite core graph with rank $\ge 1$. A Stallings subgroup graph comes with an immersion $v: S(H) \to \Gamma$, which is a restriction of the covering map. We shall also assume that $S(H)$ is subdivided so that $v$ is {\bf simplicial}: $v$ maps edges to edges. The immersion $S(H) \to \Gamma$ uniquely determines $[H]$, the conjugacy class of $H$:  if there is a subgroup $H' \le F$ with immersion $v':S(H') \to \Gamma$ and homeomorphism $h:S(H) \to S(H')$ such that $v = v'\circ h$, then $[H] = [H']$. 
\end{defn}

The last statement makes Stallings subgroup graphs particularly useful for studying a nonsurjective injective endomorphism $\phi:F \to F$. For $i \ge 1$, let $S_i = S(\phi^i(F))$. When $\phi$ is an automorphism, the maps $v_i: S_i \to \Gamma$ are all graph isomorphisms. Conversely, we get:
\begin{lem}\label{nonsurjSt} If $\phi:F \to F$ is injective and nonsurjective, then the number of edges in $S_i$ is unbounded as $i \to \infty$.
\end{lem}
\begin{proof}
Suppose there is a homeomorphism $h: S_i \to S_j$ for some $i < j$ such that $v_i = v_j \circ h$. Then, by the last statement of the definition, $[\phi^i(F)] = [\phi^j(F)]$. This means there is some element $g \in F$ such that $g\phi^i(F) g^{-1} = \phi^j(F) \le \phi^i(F) $. But in free groups (subgroup separable), a finitely generated subgroup cannot be conjugate to a proper subgroup. Hence $\phi^j(F) = \phi^i(F)$. But $\left.\phi^{j-i}\right|_{\phi^i(F)}$ is conjugate to $\phi^{j-i}$ as $\phi$ is injective. In particular, $\phi^{j-i}$ is an automorphism and so is $\phi$ -- a contradiction. Therefore, the maps $v_i$ are all distinct and the number of edges in $S_i$ is unbounded as $ i \to \infty$.
\end{proof}
We shall use this lemma in Section~\ref{immersions} to show that, for an irreducible nonsurjective endomorphism, the sequence of maps $S_i \to S_i$ that induce $\phi$ on $\pi_1(S_i) \cong F$ converges to an immersion on some graph $\Gamma'$ that induces $\phi$ on $\pi_1(\Gamma') \cong F$.

\begin{defn} Let $\Gamma$ be a connected finite core graph. A map $f: \Gamma \to \Gamma$ is a {\bf train track} if it maps vertices to vertices and all of its iterates are locally injective in interior of edges and at bivalent vertices. A train track is {\bf irreducible} if for any pair of edges $e_i$, $e_j$ in $\Gamma$, $e_i$ is in the image of some $f$-iterate of $e_j$. 
Given a train track $f: \Gamma \to \Gamma$, we fix an ordering of the edges of $\Gamma$ and construct the {\bf transition matrix} $A(f)$ as follows: it is a nonnegative square matrix whose size is given by the number of edges in $\Gamma$; the $(i, j)$-th entry of $A(f)$ is the number of times the edge $e_i$ appears in the image of $e_j$. An irreducible train track is  {\bf expanding} if it is not a homeomorphism; equivalently, the transition matrix is {\it irreducible} and has a real eigenvalue $\lambda_f > 1$ -- this is the {\it Perron-Frobenius eigenvalue}.

Let $f:\Gamma \to \Gamma$ be a train track map and $v$ be a vertex of $\Gamma$. The {\bf Whitehead graph} at $v$ is a simple graph whose vertices are the half-edges of $\Gamma$ attached to $v$, denoted by $T_v(\Gamma)$. A pair of elements of $T_v(\Gamma)$ is called a {\bf turn} at $v$ and it is {\bf nondegenerate} if the pair consists of distinct elements. The train track $f$ induces a map on $T_v(\Gamma)$. A nondegenerate turn is {\bf $f$-legal} if it remains nondegenerate under iteration of $f$. A turn at $v$ is an edge of the Whitehead graph if it appears in the image of an $f$-iterate of some edge of $\Gamma$. Note that the edges in the Whitehead graphs are $f$-legal.

An irreducible train track $f: \Gamma \to \Gamma$ is {\bf weakly clean} if the Whitehead graph at each vertex is connected. A weakly clean map is {\bf clean} if there is an iterate such that every edge surjects onto the whole graph; equivalently, the transition matrix is {\it primitive}. It follows from the definition that if a map $f$ is clean then so are all its iterates $f^i~(i \ge 1)$.
\end{defn}

The following proposition due to Dowdall-Kapovich-Leininger \cite[Proposition~B.2]{DKL15} allows us to use clean and weakly clean interchangeably. We give a proof that does not assume  homotopy-equivalence.

\begin{prop}\label{clean} Suppose $f:\Gamma \to \Gamma$ is a train track with an irreducible but not primitive transition matrix. Then $f$ has a vertex with a disconnected Whitehead graph. In particular, if $f$ is weakly clean, then $f$ is clean.\end{prop}
\begin{proof}As $A(f)$ is irreducible but not primitive, it is permutation-similar to a transitive block permutation matrix \cite[Theorem 1.8.3]{BR97}. This block permutation form gives a partition of $E(\Gamma)$ into $d \ge 2$ proper subgraphs $\Gamma_0, \ldots, \Gamma_{d-1}$ such that $f(\Gamma_i) \subset \Gamma_{i+1}$ where indices are considered$\mod d$. Thus, any vertex adjacent to two or more of these subgraphs has a disconnected Whitehead graph. Such a vertex exists since $\Gamma$ is connected.\end{proof}

\begin{defn}Two homomorphisms $\phi_1, \phi_2:A \to B$ are {\it equivalent} if there is an inner automorphism of $B$, $i_g$, such that $\phi_1 = i_g  \circ \phi_2$. An {\it outermorphism} is an equivalence class in $\operatorname{Hom}(A,B)$, denoted by $[\phi]$, and $\Out(F)$ is the group of outer automorphisms of $F$. A graph map $f:\Gamma \to \Gamma$ is a {\bf (topological) representative} for an injective (outer) endomorphism $\phi: F \to F$ if: $\Gamma$ is a connected core graph; the map $f$ maps vertices to vertices and is locally injective in interior of edges and at bivalent vertices; and there is an isomorphism $\alpha: F \to \pi_1(\Gamma)$, known as a {\bf marking}, such that $[f_*\alpha] = [\alpha\phi]$.
\end{defn}

Bestvina-Handel defined train tracks in \cite{BH92} and one of the main results was the algorithmic construction of train track representatives for irreducible endomorphisms. 
We note that their result was stated for irreducible automorphisms but the proof itself never used nor needed the fact that the endomorphisms were surjective. See also Dicks-Ventura \cite{DV96}.

\begin{thm}[{\cite[Theorem~1.7]{BH92}}]\label{tt} If $\phi: F \to F$ is an irreducible endomorphism, then $\phi$ can be represented by an irreducible train track map. The irreducible train track is expanding if and only if $\phi$ has infinite-order.\end{thm}

The following argument is now a standard technique in the field and has its roots in Bestvina-Handel's paper \cite[Proposition~4.5]{BH92}.

\begin{cor}\label{clean2} If $\phi:F \to F$ is irreducible and has infinite-order, then $\phi$ has a clean representative. Moreover, any irreducible train track representative of $\phi$ will be clean.\end{cor}
\begin{proof}[Sketch proof] By Theorem \ref{tt}, the endomorphism $\phi$ has an expanding irreducible train track representative. If the Whitehead graph of some vertex were disconnected, then blowing-up the vertex and the appropriate preimages to separate the components of the Whitehead graph would give a reduction of $\phi$. So all Whitehead graphs are connected and the representative is (weakly) clean. For more details, see \cite[Proposition 4.1]{Kap14}\end{proof}

\begin{rmk}We shall take the moment to discuss a bit of history. As we have already noted, Bestvina-Handel constructed train track representatives for irreducible automorphisms but the construction only used Stallings folds and Perron-Frobenius theory; thus it applies for irreducible endomorphisms. 

In the subsequent {\it laminations paper} \cite{BFH97}, Bestvina-Feighn-Handel show that irreducible automorphisms have weakly clean representatives. The argument again applies to irreducible endomorphisms as shown in the previous corollary and Proposition~\ref{clean} allows us to strengthen the conclusion to get clean representatives. Proposition~\ref{clean} is important as it renders the erratum \cite{BFH97e} to the laminations paper unnecessary: full irreducibility is not needed for any of the results in that paper, only the existence of a clean representative.

In fact, besides the results involving {\it repelling laminations/trees}, the rest also apply to nonsurjective endomorphisms with clean representatives; most importantly for our paper, Lemma~\ref{sink} holds in this greater generality even though it was originally stated only for (fully) irreducible automorphisms.
\end{rmk}

For any irreducible train track map $f:\Gamma \to \Gamma$, there is the associated Perron-Frobenius eigenvalue of the transition matrix, $\lambda_f \ge 1$, and a metric on the graph $\Gamma$ given by the Perron-Frobenius left eigenvector. With this metric, $f$ is homotopic rel. vertices to a (unique) map whose restrictions to edges are local $\lambda_f$-homotheties. For the rest of the paper, we shall assume an irreducible train track map has the latter linear property. The following lemmas combined with the train track linear structure allow us to carefully study the dynamics of $f$. The first lemma is known as the {\bf Bounded Cancellation Lemma}.

{~}

{\noindent \bf Notation.} We denote by $[p]$ (resp. $[\rho]$) the immersed path (resp. loop) homotopic rel. endpoints to a path $p$ (resp. freely homotopic to a loop $\rho$).
We denote the length of an immersed path $p$ in $\Gamma$ by $l(p)$. When no metric on $\Gamma$ is specified, the length is assumed to be the combinatorial length.

\begin{lem}[{\cite[Lemma II.2.4]{DV96}}]Suppose $f: \Gamma \to \Gamma$ is a topological representative for an injective endomorphism $\phi:F \to F$. Then there exists an constant $C = C(f)$ such that for any immersed path that can be written as a concatenation $a \cdot b$ in $\Gamma$, there exists path decompositions \( [f(a)] =  x \cdot u, [f(b)] =  \bar u \cdot y, [f(a \cdot b)] = x \cdot y,\) where $l(u) < C$.
\end{lem}

\begin{lem}\label{critLem}Let $f:\Gamma \to \Gamma$ be an expanding irreducible train track and $C = C(f)$ the cancellation constant. If $b$ is $f$-legal path and $l(b) > \frac{2C}{\lambda_f - 1}$, then there is a nontrivial subpath $s$ of $b$ such that $f^k(s)$ is a subpath of $[f^k(a \cdot b \cdot c)]$ for any $k \ge 1$ and immersed path $a\cdot b \cdot c$.
\end{lem}
\begin{proof}
Let $a\cdot b \cdot c$ be an immersed path. By hypothesis and bounded cancellation, $[f(a)] = x \cdot u, f(b) = \bar u \cdot y \cdot v , [f(c)] = \bar v \cdot z$ and $[f(a \cdot b \cdot c)] = x \cdot y \cdot z$ where $l(u), l(v) < C$ and $y$ is nontrivial and $f$-legal. Furthermore, $l(y) = l(f(b)) - l(u) - l(v) > \lambda_f \cdot l( b) - 2C > \frac{2C}{\lambda_f - 1}$. The subpath of $b$ corresponding to $y$ has length $> \frac{2C}{\lambda_f - 1} - \frac{2C}{\lambda_f}$. By induction, there is a nontrivial subpath $s$ of $b$ such that $f^k(s)$ is a subpath of $[f^k(a \cdot b \cdot c)]$ for any $k \ge 1$. By removing subpaths of length $\frac{C}{\lambda_f - 1}$ from the start and end of $b$, $s$ can be chosen independent of $a$ and $c$.
\end{proof}

\begin{defn}\label{crit}The number $\frac{2C}{\lambda_f - 1}$ is known as the {\bf critical constant}.
\end{defn}

\section{The Action of Injective Endomorphisms on Outer Space}\label{cvf}

Culler-Vogtmann introduced outer space in \cite{CV86} as a topological space with a nice $\Out(F)$-action. We will give two equivalent definitions and then describe the $\Out(F)$-action and, more generally, the semiaction of an injective endomorphism with respect to the two descriptions of outer space. See Karen Vogtmann's survey paper \cite{Vogt02} for details.

\begin{defn}In general, {\bf trees} refers to {\it real trees} ($0$-hyperbolic geodesic spaces) and a tree is {\bf simplicial} if it is homeomorphic to a locally finite CW-complex. 
The {\bf Culler-Vogtmann outer space} of $F$, denoted by $CV(F)$, is the set of connected simplicial trees, $T$, with a free minimal (left) action of $F$, $\alpha: F \to \mathrm{Isom}(T)$, up to the equivalence: $(T_1, \alpha_1) \sim (T_2, \alpha_2)$ if there is an $F$-equivariant homothety $f:T_1 \to T_2$, i.e., $f  \circ \alpha_1(g) = \alpha_2(g) \circ f$ for all $g \in F$. 
Alternatively, $CV(F)$ is the set of connected finite metric core graphs, $\Gamma$, with a marking $\alpha: F \to \pi_1(\Gamma)$, up to the equivalence: $(\Gamma_1, \alpha_1) \sim (\Gamma_2, \alpha_2)$ if there is a homothety $f: \Gamma_1 \to \Gamma_2$ such that $[f_* \alpha_1] = [\alpha_2]$. Identify $\pi_1(\Gamma)$ with the group of deck transformations of $\tilde \Gamma$ to get the correspondence between the two descriptions. 
\end{defn}

There are several equivalent ways of defining a topology on $CV(F)$. We need two of them for this paper; the first comes from length functions:  for any representative $(T, \alpha)$ of an equivalence class in $CV(F)$, let $l_{(T,\alpha)}: F \to \mathbb R$ be the length function that maps $g$ to its translation distance in $(T, \alpha)$. The equivalence class $[T, \alpha]$ determines a projective length function $l_{[T,\alpha]} \in \mathbb P \mathbb R^F$. The function $\iota: CV(F) \to \mathbb P \mathbb R^F$ given by $[T,\alpha] \mapsto l_{[T,\alpha]}$ is injective and the closure of its image is compact \cite{CV86}. Pullback the topology on $\mathbb{PR}^F$ via $\iota$ to get a topology on $CV(F)$. Length functions will also be useful in establishing the existence of limit trees in the next section.

The second definition of the topology is more concrete: for an equivalence class $[\Gamma, \alpha] \in CV(F)$, choose a representative such that $\Gamma$ has no bivalent vertices and $\mathrm{vol}(\Gamma) = 1$; the volume of a metric graph $\mathrm{vol}(\Gamma)$ is the sum of the lengths of all edges in the graph. Let $n$ be the number of edges in $\Gamma$ and identify $\sigma(\Gamma,\alpha)$ with the $(n-1)$-simplex one gets by varying the lengths of the edges of $\Gamma$ to get homeomorphic $\Gamma'$ while still maintaining the equality $\mathrm{vol}(\Gamma') = 1$. In the tree description of $CV(F)$, this variation corresponds to equivariantly varying the metric on $(\tilde \Gamma, \alpha)$ to get $(\tilde \Gamma', \alpha')$. 
This gives us a decomposition of $CV(F)$ into a disjoint union of open simplices $\sigma(\Gamma, \alpha)$. Attaching maps for these simplices are given by decreasing the volume of some forest of $\Gamma$ to $0$. This decomposition of $CV(F)$ and description of attaching maps makes $CV(F)$ a locally finite {\it open simplicial complex}, i.e., a simplicial complex with some missing faces corresponding to collapsing noncontractible subgraphs. The set of open simplices of $CV(F)$, denoted by $SCV(F)$, can be made into a locally finite simplicial complex known as the {\it spine of outer space} but, for the most part, we will treat it as a set with no added structure.

{~}

It is a theorem of Culler-Vogtmann that $CV(F)$ (with either of the equivalent topologies) is contractible \cite{CV86, Vogt17}. There is a natural right action of $\Out(F)$ on $CV(F)$ given by $[T, \alpha] \cdot [\phi] = [T, \alpha\phi]$ for $[T,\alpha]\in CV(F)$ and $[\phi] \in \Out(F)$. Furthermore, this action induces an action on $SCV(F)$: $\sigma(T, \alpha) \cdot [\phi] = \sigma(T, \alpha\phi)$. For an injective (outer) endomorphism $\phi:F \to F$, there is a right (semi)action given by $[T, \alpha] \cdot [\phi] = [T', \alpha \phi]$ where $T'$ is the minimal tree of the {\it twisted} action $\alpha \phi: F \to \mathrm{Isom}(T)$, i.e., $T'$ is the minimal tree for $\phi(F)$. This induces an action on $SCV(F)$ by forgetting the metrics: $\sigma(T, \alpha) \cdot [\phi] = \sigma(T', \alpha\phi)$. For an $F$-tree $(T, \alpha)$ and a marked metric graph $(\Gamma, \alpha)$, we shall abuse notation and write only $T$ and $\Gamma$ when the actions and markings are clear. 
The goal of the next section is to describe dynamics of the action of irreducible nonsurjective endomorphisms on $CV(F)$. We will show that the action has a unique attracting point and no other fixed points.

We now give the second description of the action of an injective endomorphism $\phi:F \to F$ on $CV(F)$ and $SCV(F)$. Let $\Gamma$ be a marked metric graph, i.e., $[\Gamma, \alpha] \in CV(F)$. Fix any topological representative $f:\Gamma \to \Gamma$ for $\phi$. For any $i \ge 1$, the map $f^i:\Gamma \to \Gamma$ factors  as $f^i = v_i h_i$ where $h_i: \Gamma \to S_i'$ is a composition of folds and $v_i: S_i' \to \Gamma$ is an immersion. The graph $S_i'$ is subdivided so that the immersion is simplicial. If we let $S_i \subset S_i'$ be the core subgraph, then this is the Stallings subgroup graph for $\phi^i(F)$ and the graphs fits in the following commutative diagram:

\[ \xymatrix{ 
 & S_1 \ar@{^{(}->}[d] & & S_i \ar@{^{(}->}[d] & S_{i+1} \ar@{^{(}->}[d] & \\
\Gamma \ar@{=}[d] \ar[r] & S_1' \ar[d]^{v_1} \ar[r] & \ar[r] \cdots & S_i' \ar[d]^{v_i} \ar[r] & S_{i+1}' \ar[d]^{v_{i+1}} \ar[r] & \cdots\\
\Gamma \ar[r]^f & \Gamma \ar[r]^f & \cdots \ar[r]^f & \Gamma \ar[r]^f & \Gamma \ar[r]^f & \cdots
}\]

Since $\phi$ is injective, the composition of folds $h_i: \Gamma \to S_i'$ is a homotopy equivalence which induces a marking $h_{i*} \alpha: F \to \pi_1(S_i') = \pi_1(S_i)$. Pullback the metric on $\Gamma$ via the immersion $v_i:S_i \to \Gamma$ to get a metric on $S_i$. By construction, $\tilde v_i(\tilde S_i) \subset \tilde \Gamma$ is the minimal tree for $\phi^i(F)$ and $[\tilde \Gamma] \cdot [\phi]^i = [\tilde S_i]$ as $\tilde v_i$ is an isometric embedding.  Therefore, in terms of marked metric graphs, the action is given by $[\Gamma, \alpha] \cdot [\phi]^i = [S_i, h_{i*} \alpha]$. As the immersion $v_i$ and folds $h_i$ do not depend on the metric on $\Gamma$, we see that the action of $\phi$ on $CV(F)$ is piecewise-linear with respect to the open simplicial structure and it induces the action on $SCV(F)$ by forgetting the metrics.

The following lemma tells us precisely when $\phi$ fixes an element of $SCV(F)$ or $CV(F)$. 
\begin{lem}\label{fix}Let $\phi:F \to F$ be an injective endomorphism. Then:
\begin{enumerate}
\item $\sigma(\Gamma) \in SCV(F)$ is fixed by $\phi$ if and only if $\phi$ is represented by an immersion on $\Gamma$.
\item $[\Gamma] \in CV(F)$ is fixed by $\phi$ if and only if $\phi$ is represented by a local homothety on $\Gamma$.
\end{enumerate}
\end{lem}
\begin{proof}Suppose $\sigma(\Gamma, \alpha) \cdot [\phi] = \sigma(\Gamma, \alpha)$, i.e., $\sigma(S_1, h_{i*} \alpha) = \sigma(\Gamma, \alpha)$. Then the composition of folds and core retraction $h_1: \Gamma \to S_1$ is homotopic to a homeomorphism. As $f = v_1 h_1$ and $v_1$ is an immersion, we get that $f$ is homotopic to an immersion. Suppose $[\Gamma, \alpha] \cdot [\phi] = [\Gamma, \alpha]$. By the previous sentence, we may assume $f$ is an immersion, $S_1 = S_1'$, and $h_1: \Gamma \to S_1$ is a homeomorphism. Finally, for the point $[\Gamma, \alpha]$ to be fixed, the immersion $v_1:S_1 \to \Gamma$ and homeomorphism $h_1: \Gamma \to S_1$ must induce the same projective metric on $S_1$. Thus $f$ is (homotopic to) a local homothety.
\end{proof}

We end the section with a description of the action of the endomorphism in Example~\ref{sapir}.
\begin{eg}\label{sapir2} Recall $F=F(a,b)$ and $\phi:F \to F$ is given by $(a,b)\mapsto (ab, ba)$. In rank~2, the spine of outer space has the structure of a regular trivalent tree with a spike at the midpoint of edges. Set $F = \pi_1(R_2)$, then the standard rose is the marked graph $R_* = (R_2, id_F)$ and all other roses are given by $(R_2, \varphi)$ for some $[\varphi] \in \Out(F)$. Let $(B, \beta)$ be the barbell graph attached to $R_*$ and let $(T, \theta)$ be the theta graph between $R_*,~(R_2, \varphi_a)$, and $(R_2, \varphi_b)$ where $\varphi_a:(a,ab)$ and $\varphi_b:(ba, b)$ are the generators of $\Out(F)$.

To compute $\sigma(R_2, \varphi) \cdot [\phi]$, first represent $\varphi \phi$ on $R_*$ then fold. The composition of folds gives the resulting marked graph. For instance, $\sigma(R_2, \varphi_a)\cdot[\phi] = \sigma(T, \theta) = \sigma(R_2, \varphi_b) \cdot [\phi]$, $\sigma(R_2, \varphi_a^{-1}) \cdot [\phi] = \sigma(B, \beta)$, and $\sigma(R_2, \varphi_a^{-1}) \cdot [\phi]^2 = \sigma(T, \theta \varphi_a^{-1})$. Figure \ref{fig1} below illustrates one of these computations. Along these lines, we can show that \[ S = \left\{ \sigma(R_*),~\sigma(T, \theta),~\sigma(T, \theta \varphi_a^{-1}),~\sigma(B, \beta) \right\}\] is the set of $\phi$-periodic elements in $SCV(F)$; the first two are $\phi$-fixed while the latter two have $\phi$-period 2. By inducting on the roses of the spine, we can verify $SCV(F) \cdot [\phi] = S$.
\end{eg}

\begin{figure}[h]
 \centering 
 \includegraphics[scale=0.36]{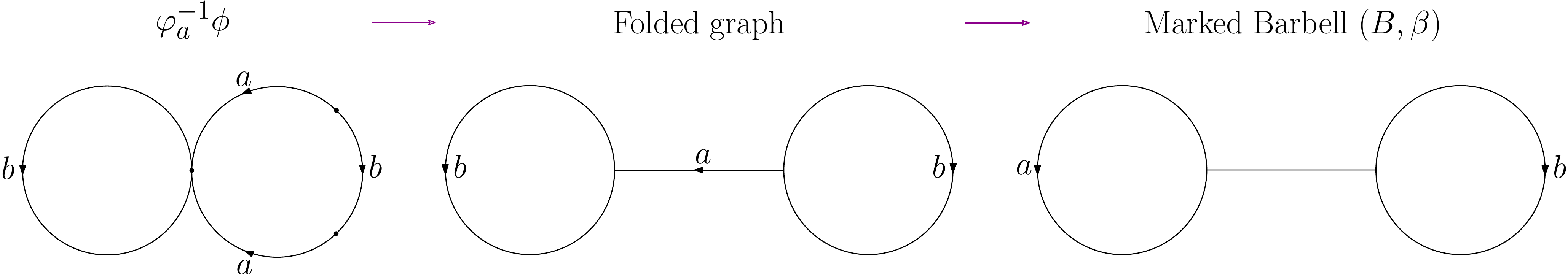}
 \caption{Computation showing $\sigma(R_2, \varphi_a^{-1}) \cdot [\phi] = \sigma(B, \beta)$.}
 \label{fig1}
\end{figure}

\section{Irreducible Nonsurjective Endomorphisms are Immersions}\label{immersions}

Given an irreducible train track map $f:\Gamma \to \Gamma$ representing an injective endomorphism $\phi:F \to F$, then the iterates of $\phi$ act on the $F$-tree $\tilde \Gamma$ by taking minimal trees of their twisted actions. If $l_{\tilde \Gamma}:F \to \mathbb R$ is the length function for $\tilde \Gamma$, then $l_{\tilde \Gamma \cdot \phi^n} = l_{\tilde \Gamma}\phi^n$. We define the limit tree $T_f$ to be the tree whose length function $l:F \to \mathbb R$ is given by \[ l = \lim_{n \to \infty} \lambda_f^{-n} \cdot \left( l_{\tilde \Gamma} \phi^n\right). \] As $f$ is an irreducible train track map, this limit exists by bounded cancellation. The $F$-action on $T_f$ is free if and only if $\phi$ is atoroidal, i.e., it has no periodic nontrivial conjugacy class/cyclic subgroup. The class $[T_f]$ need not be in $CV(F)$ and, in fact, it is not when $\phi$ is an infinite-order automorphism.

\begin{lem}[{\cite[Lemma~3.4]{BFH97}}]\label{sink}If $f:\Gamma \to \Gamma$ is a clean representative for $\phi:F \to F$, then $[T] \cdot [\phi]^n \to [T_f]$ for all $[T] \in CV(F)$.
\end{lem} 
In the lemma, the limit is taken in the space of projective length functions, $\mathbb{PR}^F$. We use $[T_\phi] = [T_f]$ and $\lambda_\phi = \lambda_f$ to emphasize independence from the choice of train track representative $f$. When $\phi$ is represented by an irreducible immersion $f:\Gamma \to \Gamma$ and $\Gamma$ has the train track linear structure, then $f$ is a local $\lambda_\phi$-homothety, $T_\phi = \tilde \Gamma$, and $[T_\phi] \in CV(F)$. In a certain sense, the immersion is unique since $[\Gamma]$ is the unique $\phi$-fixed point in $CV(F)$.

{~}

In the general setting considered in the previous section, not much can be said about the graphs $S_i$ and immersions $v_i: S_i' \to \Gamma$. However, if $f:\Gamma \to \Gamma$ is assumed to be a clean representative, we gain more control; for example, we will show in this case that $S_i'$ is a core graph, hence, $S_i = S_i'$. To do this, we need to introduce a new structure:

\begin{defn}Let $f:\Gamma \to \Gamma$ be an irreducible train track map and, for some $i \ge 1$,  $f^i = hg$ where $g:\Gamma \to X$ and $h:X \to \Gamma$ are surjective graph maps and $X$ is a finite core graph. A turn at a vertex $v$ of $X$ is {\bf legal} if its image under $h$ is an $f$-legal turn. The {\bf relative Whitehead graph} of $(X, g, h)$ at $v$ is a simple graph whose vertex set is the set of half-edges at $v$, $T_v(X)$, and a turn at $v$ is an edge of the relative Whitehead graph if it is in the edge-path $g(f^j(e))$ for some $j \ge 0$ and edge $e$ in $\Gamma$. Note that edges of the relative Whitehead graphs are legal.
\end{defn}

The next lemma tells us how subdivision and folding affect the relative Whitehead~graphs.
\begin{lem}\label{rel1}Suppose $f^i: \Gamma \xrightarrow{g} X \xrightarrow{h} \Gamma $ has connected relative Whitehead graphs and $h=h_2 h_1$ where $h_1: X \to Y$ is a subdivision or a single fold. Then $Y$ is a core graph and the relative Whitehead graphs of $(Y, h_1 g, h_2)$ are connected.
\end{lem}
\begin{proof}Suppose $h_1$ is a subdivision. If $v$ is a vertex of $Y$ that was added by the subdivision, then $T_v(Y)$ has two elements, $v \in \mathrm{Int}(e)$ for some edge $e'$ of $X$, and $v \in h_1(e')$. In particular, the two elements of $T_v(Y)$ are joined by an edge since $g$ is surjective.

If $v$ is a vertex of both $Y$ and $X$, then $T_v(Y) = T_v(X)$ and subdivision makes no changes to the relative Whitehead graph at $v$. So the relative Whitehead graph of $(Y, h_1g, h_2)$ is connected since we assumed $(X, g, h)$ has connected relative Whitehead graphs.

Suppose $h_1$ is a single fold. The only way to produce a univalent vertex in $Y$ with a fold is by folding a bivalent vertex of $X$. But bivalent vertices only have one turn which is legal since the relative Whitehead graphs of $(X, g, h)$ are connected.

At the origin vertex, the fold identifies two distinct vertices of the relative Whitehead graphs and this preserves connectedness. At the two terminal vertices, the fold identifies two distinct vertices, one from each relative Whitehead graph, and this gluing of two connected graphs produces a connected graph. If a vertex in $X$ is not part of the fold, then its image in $Y$ will have the same connected relative Whitehead graph.
\end{proof}

Recall that to construct $S_i'$, we subdivide and fold the map $f^i:\Gamma \to \Gamma$ to get the immersion $v_i: S_i' \to \Gamma$. Suppose $f$ was a clean map. For the base case, the relative Whitehead graphs of $(\Gamma, id_\Gamma, f^i)$ are connected as they are isomorphic to the Whitehead graphs of $f^i$. By Lemma \ref{rel1} and induction on folds, $S_i'$ is a core graph and the relative Whitehead graphs of $(S_i, h_i, v_i)$ are connected. So the immersion $v_i$ maps edge-paths connecting branch points ({\it natural edges}) of $S_i$ to $f$-legal edge-paths in $\Gamma$. The commutative diagram in the previous section becomes:
\[\xymatrix{
\Gamma \ar[r] \ar@{=}[d] & S_1 \ar[r] \ar[d]^{v_1} & S_2 \ar[r] \ar[d]^{v_2} & \cdots \ar[r] & S_i \ar[r] \ar[d]^{v_i} & S_{i+1} \ar[r] \ar[d]^{v_{i+1}} & \cdots \\
\Gamma \ar[r]^f & \Gamma \ar[r]^f & \Gamma \ar[r]^f & \cdots \ar[r]^f &  \Gamma \ar[r]^f & \Gamma \ar[r] & \cdots
}\]
Aside from the construction of this diagram and the next lemma, this section and the next are independent of each other. This lemma will only be used in the proof of Proposition~\ref{invSbgrp}, the main result of the next section.
\begin{lem}\label{rel2}Let $f, v_i, h_i$ be as in the previous paragraph. The map $\hat f^i: S_i \to S_i$ given by $\hat f^i = h_i v_i$ is a clean representative for $\phi^i$.
\end{lem}
\begin{proof} For all $n \ge 1$, $(\hat f^i)^{n+1} = (h_i v_i)^{n+1} = h_i ( v_i h_i )^{n} v_i = h_i f^{in} v_i $. Since $f$ is a train track, $v_i$ is a simplicial immersion, and $h_i$ folds illegal turns only, it follows that $\hat f^i$ is a train track too. Since $f$ is clean,  there is an $n \ge 1$ such that, for any edge $e$ in $\Gamma$, $f^{in}(e)$ is surjective and contains all the edges of the Whitehead graphs of $f^i$. As $h_i$ is surjective, $(\hat f^i)^{n+1}$ is surjective when restricted to any edge and $\hat f^i$ has a primitive transition matrix. Since folding only identifies vertices of relative Whitehead graphs, $h_i f^{in}(e)$ contains all the edges of the relative Whitehead graphs of $(S_i, h_i, v_i )$. As $v_i$ maps edges to edges, we have that the Whitehead graphs of $\hat f^i$ contains the relative Whitehead graphs of $(S_i, h_i, v_i)$ as subgraphs. Since $v_i$ is surjective, we have the reverse containment too. So the Whitehead graphs of $\hat f^i$ are connected and the map is clean.
\end{proof}

With this set-up, we are now ready to present a new proof of an unpublished theorem due to Patrick Reynolds \cite[Corollary~5.5]{Rey}. The first half of this proof is based on an argument due to Ilya~Kapovich showing that irreducible nonsurjective endomorphisms are expansive; c.f. \cite[Proposition~3.11]{Rey}.

\begin{thm}\label{part2}If $\phi:F\to F$ is an irreducible nonsurjective endomorphism, then it can be represented by a clean immersion, unique amongst irreducible immersions.\end{thm}
\begin{proof} 
Let $f:\Gamma \to \Gamma$ be a clean representative for $\phi: F \to F$ given by Corollary~\ref{clean2}. By the discussion following Lemma~\ref{rel1}, we now construct core graphs $S_i$, folds $h_i:\Gamma \to S_i$, and immersions $v_i: S_i \to \Gamma$ such that $f^i = v_i h_i$ and the natural edges of $S_i$ are legal, i.e., they consist of legal turns. By Lemma~\ref{nonsurjSt}, nonsurjectivity implies that the number of edges in $S_i$ is unbounded as $i \to \infty$. So for some $i \gg 0$, $S_i$ has a natural edge longer than the critical constant. By Lemma~\ref{critLem}, there is a nontrivial subpath $s$ of the long natural edge such that, for all $ j > i$, the folds in $S_i \to S_j$ are supported in the complement of $s$.

Fix an index $i \gg 0$ such that $S_i$ has the maximal number of natural edges longer than $\frac{2C\cdot L \cdot\lambda_\phi}{\lambda_\phi - 1}$, where $L = L( \mathrm{rank}(F))$ is the maximum number of natural edges of an embedded loop of a marked graph. Denote by $K_i$ the proper subgraph consisting of the remaining {\it short} natural edges of $S_i$ with length at most the critical constant. Let $\hat f: S_i \to S_i$ be a map representing $\phi$, i.e., $[\hat f_* h_{i*}] = [h_{i*} f_*]$. Suppose $K_i$ has an embedded loop $\rho$, then by construction $f v_i(\rho)$ decomposes into at most $L$ legal paths, each of length at most $\frac{2C \lambda_\phi}{\lambda_\phi - 1}$. So $[f v_i(\rho)]$ lifts to the loop $[\hat f(\rho)]$ in $K_i$. Thus if $K_i$ is not a forest, then $K_i$ is an $\hat f_i$-invariant proper subgraph up to homotopy. Since $\phi$ is irreducible,  $K_i$ must be a forest. Therefore, the loops in $S_i$ are growing exponentially, $\phi$ is atoroidal, and the action of $F$ on $T_\phi$ is free.

We showed in the previous paragraph that folds in the map $S_i \to S_j~(j > i)$ are supported in some forest $K \subset S_i$. Since there are finitely many combinatorially distinct ways to fold a forest, there are fixed $j > i$ and $k \ge 1$ such that the composition of folds $S_j \to S_{j+k}$ is homotopic to a homeomorphism. 
This means $[\tilde S_j], [\tilde S_{j+k}]$ lie in the simplex $\sigma(\tilde S_j) \subset CV(F)$. In particular, $\sigma(\tilde S_j) \cdot [\phi]^k = \sigma(\tilde S_j)$ and the sequence $[\tilde S_{j + kn}] ~ (n \ge 0)$ lies in $\sigma(\tilde S_j)$.

By definition of the limit tree $T_f$ and the action $\tilde \Gamma \phi^n = \tilde S_n$, we get: \[ T_f = \lim_{n \to \infty} \lambda_f^{-n} \tilde \Gamma \phi^n = \lim_{n \to \infty} \lambda_f^{-j-kn} \tilde \Gamma \phi^{j+kn} = \lim_{n \to \infty} \lambda_f^{-j-kn} \tilde S_{j+kn}.\]
Therefore, the limit tree is simplicial since $[T_f]$ is in the (simplicial) closure of $\sigma(\tilde S_j)$, i.e., the $F$-quotient of $T_f$, call it $\Gamma_0$, is a graph obtained by collapsing a subgraph of $S_j$ and rescaling the metric. As $T_f$ is free and simplicial, $[T_f] \in CV(F)$. By definition and Lemma~\ref{fix}, $[T_f]$ is fixed by $\phi$ and $\phi$ is represented by a local $\lambda_\phi$-homothety $f_0:\Gamma_0 \to \Gamma_0$. Furthermore, the local homothety $f_0$ is expanding since $\lambda_\phi > 1$; so $f_0$ has no invariant forests and the irreducibility of $\phi$ implies $f_0$ is clean by Corollary~\ref{clean2}. Uniqueness follows from Lemma~\ref{fix} and uniqueness of the fixed point $[T_\phi]$ due to Lemma~\ref{sink}.
\end{proof}

This  theorem tells us that the action of an irreducible nonsurjective endomorphism on $CV(F)$ has a unique fixed point that is also a global attracting point. Reynolds studied this action further; for instance, the action converges to the fixed point uniformly on compact sets and, for {\it admissible} endomorphisms, the action can be extended to the compactification of outer space $\overline{CV(F)}$ where the convergence is uniform.

We now use the action on outer space again to prove a partial converse to Theorem~\ref{part2}, which can also be thought of as a criterion for fully irreducible nonsurjective endomorphisms.

\begin{thm}\label{part3}If $\phi:F \to F$ is represented by a clean immersion whose Whitehead graphs have no cut vertices, then $\phi$ is nonsurjective and fully irreducible.\end{thm}
\begin{proof} Since $\phi$ is represented by a clean immersion $f:\Gamma \to \Gamma$, it is nonsurjective. Suppose $\phi^n(A) \le gAg^{-1}$ where $A \le F$ is a proper free factor, $g \in F$, and $n \ge 1$. Let $S(A)$ be the Stallings subgroup graph corresponding to $A$ with respect to $\Gamma$. Set $\psi = i_g \circ \phi^n$ so that $\psi(A) \le A$ and the immersion $f^n$ lifts to an immersion $g: S(A) \to S(A)$ representing $\left.\psi\right|_A$. Complete $\Delta_0 = S(A)$ to a graph $\Gamma'$ with a marking $\pi_1(\Gamma') \cong F$ and extend $g$ to the rest of $\Gamma'$ such that $g: \Gamma' \to  \Gamma'$ is a topological representative for $\psi$ and $\Delta_0$ is a noncontractible $g$-invariant proper subgraph corresponding to $A$.

Recall from Section \ref{cvf}, the Stallings subgroup graph $S_i = S(\psi^i(F))$ with respect to $\Gamma'$ along with the marking given by folding $g^i: \Gamma' \to \Gamma'$ determine the $i$-th iterate of $[\Gamma'] \in CV(F)$ under the $\psi$-action, i.e., $[\Gamma'] \cdot [\psi]^i = [S_i]$. Since $[S_i] \to [\Gamma]$ by Lemma~\ref{sink} and the spine of outer space is locally finite, we have that the sequence $\sigma(S_i) ~(i \ge 1)$ in $SCV(F)$ is eventually periodic. So for some fixed $i, j \ge 1$, we have $\sigma(S_i) = \sigma(S_{i+j})$ and, by Lemma \ref{fix}, $\psi^j$ can be represented by an immersion on $S_i$.

Let $h_i:\Gamma' \to S_i'$ be the composition of folds given by the construction of $S_i$. Since $\left.g\right|_{\Delta_0}$ is an immersion, the corresponding restriction $\left.h_i\right|_{\Delta_0}$ is either a homeomorphism or an identification of vertices of $\Delta_0$ possibly followed by folds. 
Set $\Delta_i = h_i(\Delta_0)$ and let $g_i^j: S_i \to S_i$ be the induced map representing 
$\psi^j$. Then $\Delta_i$ is a noncontractible $g_i$-invariant subgraph and $\left.g_i\right|_{\Delta_i}$ is an immersion. By the previous paragraph, $g_i^j$ is homotopic to an immersion 
$\gamma: S_i \to S_i$. The homotopy will preserve the invariance of $\Delta_i$ so that $\Delta_i$ is a $\gamma$-invariant subgraph. As $\sigma(S_i)$ is fixed by $\psi^j$, the sequence $S_{i+jm}$ is constructed by pulling back the metric of $S_i$ via $\gamma^m$.

Iteratively pulling back the metric via $\gamma$  and normalizing the metric has the effect of collapsing $\gamma$-invariant forests so that the induced map is a local homothety. By uniqueness of the limit $[S_{i+jm}] \to [\Gamma]$, the induced map must be $f^{nj}$ and $\Delta_\infty$, the image of $\Delta_i$ under the collapse map, is an $f^{nj}$-invariant subgraph.
But $f^{nj}$ is a clean map, so $\Delta_\infty = \Gamma$. 

Let $h_\infty: \Delta_0 \to \Delta_\infty = \Gamma$ be the induced map. By construction, $ f^{nj} \circ h_\infty = \left.h_\infty \circ g^j \right|_{\Delta_0}$ and, as $h_\infty$ is an identification of some vertices possibly followed by a folding and/or a collapse of a forest, the Whitehead graphs of $f^{nj}$ are determined by where $g^j$ maps the edges of $\Delta_0$. So $f^{nj}$ will have disconnected Whitehead graphs or Whitehead graphs with cut vertices (depending on folds in $h_\infty$) at the identified vertices -- a contradiction.
\end{proof}

\begin{eg}\label{countereg} Let $F = F(a,b,c)$ and $\phi:F \to F$ is given by $(a,b,c) \mapsto (aba, c^2, cabac)$. The obvious topological representative on the standard marked rose is a clean map. Furthermore, the Stallings subgroup graphs $S_i = S(\phi^i(F))$, for $i \ge 1$, all determine the same marked graph (modulo metric), i.e., the same vertex in the spine. In particular, we can verify that $\phi$ is induced by a clean immersion on $S_1$. However, $\phi$ is reducible; in fact, $\phi(F) \le \langle aba, c \rangle = A$ and the latter is proper free factor of $F$.

Furthermore, the restriction $\phi_A$ is represented by a clean immersion on a rose whose Whitehead graph has a cut vertex. However, by Theorem~\ref{hypthm} below, the {\it mapping torus} of $\phi_A$ is word-hyperbolic and, as $A$ has rank 2, the endomorphism $\phi_A$ is fully irreducible. Thus, the lack of cut vertices in Theorem~\ref{part3} is not a necessary condition.
\end{eg}

It would be interesting to find an algorithmic characterization of fully irreducible nonsurjective endomorphisms.

\section{Subgroups Invariant Under Irreducible Endomorphisms}\label{iwip}

In this section, we generalize a result by Bestvina-Feighn-Handel \cite[Proposition~2.4]{BFH97}  and I.~Kapovich~\cite[Proposition~4.2]{Kap14} that characterizes the finitely generated subgroups of $F$ that support a {\em leaf of the lamination} of an irreducible automorphism.

\begin{defn}\label{leaf} Given an expanding irreducible train track map $f:\Gamma \to \Gamma$, the {\bf (attracting) lamination of $f$}, denoted by $\Lambda(f)$, is defined by iterated neighbourhoods of $f$-periodic non-vertex points. Precisely, suppose $x \in \mathrm{Int}(e)$ for some $e \in E(\Gamma)$ and $k \ge 1$ are such that $f^k(x) = x$. Then a {\bf leaf of the lamination of $f$}, is the isometric immersion $\gamma_x:\mathbb R \to \Gamma$ such that $\gamma_x(0)=x$ and $f^k(\gamma_x(r)) = \gamma_x(\lambda_f^k \cdot r)$ for all $r \in \mathbb R$, unique up to orientation.
\[ \Lambda(f) = \left\{\gamma_x \,:\, x \text{ is an $f$-periodic non-vertex point.} \right\} \]
For any integer $m \ge 1$, $f^m$-periodic points are $f$-periodic and vice-versa. So it follows that $\Lambda(f) = \Lambda(f^m)$ for any $m \ge 1$.

We shall say a nontrivial subgroup $H \le F$ {\bf supports a leaf of } $\Lambda(f)$ if some leaf $\gamma_x \in \Lambda(f)$ has a lift $\hat \gamma_x: \mathbb R \to S(H)$ to the Stallings subgroup graph of $H$.
\end{defn}

We now address the difficulty that arises when generalizing Bestvina-Feighn-Handel and I.~Kapovich's results. Any automorphism $\phi: F \to F$ permutes the finite index subgroups with the same index; so given any finite index subgroup $H'$, there exists $i \ge 1$ such that $\phi^i(H') = H'$. This fact is used to lift a clean map representing $\phi$ to a clean map on the Stallings subgroup graph $S(H')$. However, this fails when dealing with nonsurjective injective endomorphisms, i.e., there is no reason why some $\phi$-iterate of $H'$ must be a subgroup of $H'$. The next lemma gives us a way of getting around this failure. The key observation is to look at backward iteration: when $\phi$ is injective, then pre-images of finite index subgroups are finite index subgroups with the same index or less.

\begin{lem}\label{lemFI} 
Let $\phi : F \to F$ be an injective endomorphism of a free group $F$. If $H' \le F$ is a finite index subgroup, then there exist $j > k \ge 0$ such that $\phi^{-k}(H') = \phi^{-j}(H')$. 

Furthermore, for $K = \phi^{-k}(H')$, there is an induced set bijection $\varphi: F/K \to F/K$ such that the following diagram commutes:
\[ \xymatrix{
F \ar[d]^\pi \ar[r]^{\phi^{j-k}} &F\ar[d]^\pi\\
F/K \ar[r]^\varphi &F/K}\]
\end{lem}
\begin{proof} Set $H_k = \phi^{-k}(H')$ for $k \ge 0$. Then $\phi^k(H_k) = \phi^k(F)\cap H'$. Since $H'$ has finite index in $F$, then $\phi^k(H_k) = \phi^k(F)\cap H'$ has finite index in $\phi^k(F)$. In fact, $[F:H'] \ge [\phi^k(F):\phi^k(H_k)]$. But $\phi^k: F \to \phi^k(F)$ is an isomorphism that maps $H_k \le F$ to $\phi^k(H_k) \le \phi^k(F)$. Thus, $[\phi^k(F):\phi^k(H_k)] = [F:H_k]$ and $H_k$ has finite index in $F$ bounded by $[F:H']$ for all $k \ge 0$. As there are only finitely many subgroups with index bounded by $[F:H']$, there must exists $j > k \ge 0$ such that $H_k = H_j$, i.e., $\phi^{-k}(H') = \phi^{-j}(H')$. 

Let $K = \phi^{-k}(H')$ and $\pi:F \to F/K$ be the (left) coset projection map. The function $\pi \circ \phi^{j-k}$ factors through $\pi$ if and only if $\phi^{j-k}(K) \le K$, or equivalently, $K \le \phi^{k-j}(K)$, and the induced function $\varphi:F/K \to F/K$ is an injection if and only if $\phi^{k-j}(K) \le K$. By construction, $\phi^{k-j}(K)=K$, i.e., both conditions are satisfied, so $\pi \circ \phi^{j-k} = \varphi \circ \pi$ where $\varphi$ is a bijection since it is an injection of the finite set $F/K$ into itself.
\end{proof}

The main result of the section follows:

\begin{prop}\label{supSbgrp} Suppose $\phi : F \to F$ is injective and represented by a clean map. If $H \le F$ is a nontrivial finitely generated subgroup that supports a leaf of $\Lambda(f)$, then \[ [\phi^k(F) : \phi^k(F) \cap H] < \infty \quad \text{for some } k \ge 0. \]
\end{prop}
\begin{proof}
As $H$ is nontrivial and finitely generated, $S(H)$ is a nontrivial finite graph and we can add edges to $S(H)$ if necessary to extend the immersion $S(H) \to \Gamma$ to a finite cover $S(H') \to \Gamma$ corresponding to $H \le H' \le F$. Thus $[F:H'] < \infty$ and we can apply Lemma \ref{lemFI} to get $j > k \ge 0$ such that $\phi^{-k}(H') = \phi^{-j}(H')$. Set $K = \phi^{-k}(H')$. Then $\phi^k(K) = \phi^k(F) \cap H'$ has finite index in $\phi^k(F)$. Recalling the diagram preceding Lemma~\ref{rel2}, the graph $S(\phi^k(F)) = S_k$ and the cover corresponding to $\phi^k(K) \le \phi^k(F)$ lie in the following commutative diagram:
\[\xymatrix{
 S(\phi^k(K)) \ar[d]^p  & S(\phi^k(K)) \ar[d]^p \\
S(\phi^k(F)) \ar[d]^{v_k} \ar@{-->}[r]^{g} &  S(\phi^k(F)) \ar[d]^{v_k} \\
\Gamma \ar[ru]_{h_k } \ar[r]_{f^k} & \Gamma
}\]
where $g: S(\phi^k(F)) \to S(\phi^k(F))$ is clean by Lemma \ref{rel2} and $p: S(\phi^k(K)) \to S(\phi^k(F))$ is a finite cover.
The map $h_k$ maps $f^k$-periodic points to $g$-periodic points and $v_k$ maps $g$-periodic points to $f^k$-periodic points. This allows us to identify $\Lambda(f)$ with $\Lambda(g)$.

Let $\Delta_H = S(\phi^k(F) \cap H)$. This is a nontrivial graph since $\phi(H) \le H$. Note that $ S(\phi^k(K))$ is the pullback of $ S(\phi^k(F)) \to \Gamma$ and $S(H') \to \Gamma$ and $\Delta_H$ is the pullback of $ S(\phi^k(F)) \to \Gamma$ and $S(H) \to \Gamma$. As $S(H)$ is a subgraph of  $S(H')$, we have that $\Delta_H$ is a subgraph of $ S(\phi^k(K))$. The graphs fit in this commutative diagram:
\[\xymatrix{
\Delta_H \ar[d] \ar@{^{(}->}[r] & S(\phi^k(K)) \ar[d] \ar[r]^p &  S(\phi^k(F)) \ar[d]^{v_k} \\
S(H) \ar@{^{(}->}[r] & S(H') \ar[r] &\Gamma
}\]

Since $ S(\phi^k(F))$ and $S(H)$ support a leaf of $\Lambda(f)$, it follows that their pullback $\Delta_H$ supports a leaf of $\Lambda(f)$. Following the identification $\Lambda(f) \cong \Lambda(g)$, we have that $\Delta_H$ supports a leaf of $\Lambda(g)$.
Lemma~\ref{lemFI} says $\phi^{j-k}$ induces a permutation of $F/K \cong \phi^k(F)/\phi^k(K)$. Therefore, the map $g^{j-k}:S(\phi^k(F))\to S(\phi^k(F))$ lifts to a map $\hat g: S(\phi^k(K))\to S(\phi^k(K))$ with the property: if $g^{j-k}(x) = x$, then $\hat g$ permutes the elements of $p^{-1}(x)$.

The rest of the argument follows that of Bestvina-Feighn-Handel \cite[Lemma~2.1]{BFH97}. We include the details for completeness; see also \cite[Proposition~4.2]{Kap14}.

\begin{claim}The map $\hat g: S(\phi^k(K)) \to  S(\phi^k(K))$ is clean.
\end{claim}
\begin{proof}Let $\{a', b'\}$ be a turn at a vertex $v \in  S(\phi^k(K))$ such that its projection under $p$, $\{a, b\}$, is an edge of the Whitehead graph of $g^{j-k}$ at $p(v)$. Since $g^{j-k}$ is a clean map, we can replace it with an iterate and assume $g^{j-k}(a) = \ldots ab \ldots$. So $a$ contains a $g^{j-k}$-fixed point $x$ and, consequently, $\hat g$ permutes the lifts $p^{-1}(x)$. Replace $g^{j-k}$ with an iterate if necessary and assume $\hat g$ fixes $p^{-1}(x)$ and let $x' \in p^{-1}(x)$ be the lift of $x$ in $a'$. Then $x'$ is a $\hat g$-fixed point and $\hat g(a') = \ldots a'b' \ldots$. An identical argument shows that $\hat g(b') = \ldots a'b' \ldots$ after passing to a power if necessary. Thus the Whitehead graph at $v$ with respect to $\hat g$ is isomorphic to the Whitehead graph at $p(v)$ with respect to $g^{j-k}$. So the Whitehead graphs of $ S(\phi^k(K))$ with respect to $\hat g$ are connected and $\hat g$ is a train track as the turns in $\hat g(e')$ are lifts of the turns in $g^{j-k}(p(e'))$ for any edge $e$ in $ S(\phi^k(K))$. It remains to show that $g$ is irreducible.

Since the Whitehead graphs of $ S(\phi^k(K))$ are connected, for any edges $a, b$ in $ S(\phi^k(K))$, there is a sequences of turns $\{\epsilon_1, \epsilon_2\}, \{\epsilon_3, \epsilon_4\}, \ldots, \{\epsilon_{2l-1} \epsilon_{2l}\}$ such that each turn is an edge of the corresponding Whitehead graph, $\epsilon_{2m}, \epsilon_{2m+1}$ are half-edges of the same edge for $1 \le m < l$, and $\epsilon_1, \epsilon_{2l}$ are half-edges of $a, b$ respectively. Let $a=e_0, e_1, \ldots, e_{l-1}, e_l = b$ the corresponding sequence of edges. By the previous paragraph, some iterate of $\hat g$ maps $e_m$ to $e_{m+1}$ for $0 \le m < l$.  By induction, some iterate of $\hat g$ maps $a$ to $b$. As $a$ and $b$ were arbitrary, $\hat g$ is an irreducible train track.
\end{proof}
\begin{claim}The subgraph $\Delta_H \subset  S(\phi^k(K))$ is not proper, i.e., $\Delta_H =  S(\phi^k(K))$.
\end{claim}
\begin{proof}Recall that a leaf of $\Lambda(g) = \Lambda(g^{j-k})$ is constructed by iterating neighbourhoods of some $g^{j-k}$-periodic non-vertex point. Suppose $x \in \mathrm{Int}(e)$ for some edge $e$ in $S(\phi^k(F))$ and $l \ge 1$ are such that $g^{(j-k)l}(x) = x$ and $g^{(j-k)l}$ lifts to a map $\hat g^l: S(\phi^k(K)) \to  S(\phi^k(K))$ that fixes $p^{-1}(x)$ pointwise. Then the lift of $\gamma_x \in \Lambda(g^{j-k})$ to $\Delta_H$ can be constructed by using $\hat g^l$ to iterate a neighbourhood of some $x' \in p^{-1}(x) \cap \Delta_H$. Let $e'$ be the edge containing $x'$. Since $\hat g$ is clean, $\hat g^m(e')$ surjects onto the whole graph $S(\phi^k(K))$ for some $m \ge 1$ and any iterated neighbourhood of $x'$ shall eventially cover $ S(\phi^k(K))$. Therefore, $S(\phi^k(K)) \subset \Delta_H$.
\end{proof}
So the natural map $\Delta_H \to S(\phi^k(F))$ is a finite cover and $[\phi^k(F) : \phi^k(F) \cap H] < \infty$.
\end{proof}

In this paper, we will need the conclusion of the previous proposition to hold for {\it invariant} subgroups:

\begin{prop}\label{invSbgrp} Suppose $\phi : F \to F$ is injective and represented by a clean map. If $H \le F$ is a finitely generated subgroup such that $\phi(H) \le H$ and $H$ contains a $\phi$-expanding conjugacy class, then \[ [\phi^k(F) : \phi^k(F) \cap H] < \infty \quad \text{for some } k \ge 0. \]
\end{prop}
\begin{proof}
Suppose $\phi(H) \le H$ for some finitely generated group $H \le F$. Let $f: \Gamma \to \Gamma$ be the given clean map and $\rho$ be the $f$-expanding immersed loop in $\Gamma$ that lifts to a loop in $S(H)$. The invariance $\phi(H) \le H$ implies the loops $[f^{k}(\rho)]$ lift to loops in $S(H)$ for all $k \ge 1$. As $\rho$ is $f$-expanding and $f$ is a clean map, length of $[f^{k}(\rho)]$ grows arbitrarily with $k$ while the number of $f$-illegal turns in $[f^{k}(\rho)]$ remains bounded. Thus, for some $k \ge 1$, the loop $[f^{k}(\rho)]$ will contain an $f$-legal subpath longer than the critical constant. By Lemma~\ref{critLem}, there is a nontrivial subpath $s$ of $\rho$ such that $f^{k}(s)$ is a subpath of $[f^{k}(\rho)]$ for all $k \ge 1$. As $f$ is clean, some $f$-iterate of $s$ maps onto $\Gamma$.

Let $x$ be an $f$-periodic non-vertex point of $\rho$ and $\gamma_x: \mathbb R \to \Gamma$ be the corresponding leaf of $\Lambda(f)$. For all real $r > 0$, the path $\gamma_x([-r,r])$ is a subpath of $[f^{k}(\rho)]$ for some $k \ge 1$ since $x$ is contained in some $f$-iterate of $s$. Thus $\left.\gamma_x\right|_{[-r, r]}$ has a lift $\gamma_{x,r} : [-r, r] \to S(H)$. There are only finitely many preimages of $x$ in $S(H)$, so after passing to an unbounded increasing subsequence $r_m$, we can assume the sequence $(\gamma_{x,r_m}(0))_{m=1}^\infty$ is constant. By uniqueness of lifts, $\gamma_{x,r_{m+1}}$ is an extension of $\gamma_{x,r_{m}}$ for $m \ge 1$. The limit immersion $\gamma_{x, \infty}: \mathbb R \to S(H)$ is a lift of $\gamma_x$ and so $H$ supports a leaf of $\Lambda(f)$. The conclusion follows from Proposition~\ref{supSbgrp}.
\end{proof}

We shall now give necessary and sufficient conditions for an endomorphism to be fully irreducible. The reader only interested in our hyperbolicity result can skip to the next section.

\begin{thm}\label{propKap} Let $\phi:F \to F$ be an injective endomorphism. Then $\phi$ is fully irreducible if and only if $\phi:F \to F$ has no periodic cyclic free factor, is represented by a clean map,  and its image $\phi(F)$ is not contained in a proper free factor.
\end{thm}
\begin{proof} The forward direction follows from the definition of fully irreducible and Corollary~\ref{clean2}. We now prove the reverse direction.
Let $A \le F$ be a $\phi^n$-invariant free factor for some $n \ge 1$ and assume $A$ has minimal rank (for all $n$). Then $i_g \circ \phi^n(A) \le A$ for some inner automorphism $i_g$. If $A$ is cyclic then it is generated by an element in a $\phi$-expanding conjugacy class since it cannot be periodic. Otherwise, the minimality assumption implies $\left.i_g \circ \phi^n\right|_A$ is fully irreducible. In both cases, $A$ contains a $\phi$-expanding conjugacy class. As $\phi$ is represented by a clean map, so is $i_g \circ \phi^n$. By Proposition~\ref{invSbgrp}, there exists $k \ge 0$ such that $ (i_g\circ\phi^n)^k(F) \cap A$ has finite index in $(i_g \circ \phi^n)^k(F)$. It remains to show $A = F$, hence, $\phi$ is fully irreducible.

Since $A$ is a free factor of $F$,  the intersection $ (i_g\circ\phi^n)^k(F) \cap A$ is a free factor of $(i_g \circ \phi^n)^k(F)$. The only finite index free factor of a finitely generated free group is the free group itself. Thus $ (i_g \circ \phi^n)^k(F) \cap A = (i_g \circ \phi^n)^k(F)$ and $(i_g \circ \phi^n)^k(F)\le A$. 

Recall that $\phi(F)$ is not contained in a proper free factor. Suppose $\phi^{l-1}(F)$ is not contained in a proper free factor of $F$ for some $l \ge 2$. Let $K \le F$ be a free factor such that $\phi^l(F) \le K$. Then $\phi(\phi^{l-1}(F))=\phi^l(F) \le K \cap \phi^{l-1}(F)$ and the intersection is a free factor of $\phi^{l-1}(F)$. As $\phi^{l-1}$ is injective, we have $\phi|_{\phi^{l-1}(F)}$ is conjugate to $\phi$. Therefore, $\phi^l(F)$ is not contained in a proper free factor of $\phi^{l-1}(F)$ and $K \cap \phi^{l-1}(F) = \phi^{l-1}(F)$. Therefore $\phi^{l-1}(F)  \le K$. The induction hypothesis is that $\phi^{l-1}(F)$ is not contained in a proper free factor of $F$, hence $K= F$. So $\phi^l(F)$ is not contained in a proper free factor of $F$. 

By induction, $\phi^n(F)$ is not contained in a proper free factor of $F$. We also have $i_g \circ \phi^n(F)$ is not contained in a proper free factor of $F$ since $i_g$ is an automorphism. By the same induction argument, $(i_g \circ \phi^n)^k(F)$ is not contained in a proper free factor of $F$. Therefore, $(i_g \circ \phi^n)^k(F) \le A$ implies $A=F$ and $\phi$ is fully irreducible.
\end{proof}

Example~\ref{countereg} shows that the condition on the image of the endomorphism is not redundant. This proposition extends the characterization of fully irreducible automorphisms due to I.~Kapovich \cite[Theorem~1.2]{Kap19}, which in turn was motivated by Catherine Pfaff's criterion for irreducibility \cite[Theorem~4.1]{Pfaff}.

A result due to Dowdall-Kapovich-Leininger is that atoroidal irreducible automorphisms are fully irreducible \cite[Corollary~B.4]{DKL15}. As a corollary of the characterization, we get the equivalence for irreducible nonsurjective endomorphisms.

\begin{cor}\label{myequiv} If $\phi:F \to F$ is irreducible and nonsurjective, then it is fully irreducible.
\end{cor}
\begin{proof}
By Corollary~\ref{clean2}, it is represented by a clean map. By the first half of the proof of Theorem~\ref{part2}, $\phi$ is atoroidal. In particular, it has no periodic cyclic free factor. Finally, irreducibility implies its image $\phi(F)$ is not contained in a proper free factor. The conclusion follows from Theorem~\ref{propKap}.
\end{proof}

\section{Irreducible Nonsurjective Endomorphisms are Hyperbolic}\label{hyperbolic}

The goal of this final section is to prove that the mapping tori of irreducible nonsurjective endomorphisms are word-hyperbolic.

\begin{defn}Let $\phi: F \to F$ be an injective endomorphism. Then the {\bf ascending HNN extension/mapping torus} of $\phi$ is given by the presentation:
\[ F*_\phi = \langle F, t~ | ~ t^{-1}x t = \phi(x), \forall x \in F \rangle \]
\end{defn}

Thurston's hyperbolization theorem gives the correspondence between the geometry of $3$-manifolds that fiber over a circle and the dynamics of their monodromies \cite{Thu82} and Brinkmann generalized this to free-by-cylic groups $F \rtimes \mathbb Z$ \cite{Bri00}. The following theorem, the main result of \cite{JPM}, is a  partial generalization to ascending HNN extensions $F*_\phi$.

\begin{thm}[{\cite[Theorem~6.3]{JPM}}]\label{mythm}Suppose $\phi:F \to F$ is represented by an immersion. Then $F*_\phi$ is word-hyperbolic if and only if there are no $d,n \ge 1$ and $1 \neq a \in F$ such that $\phi^n(a)$ is conjugate to $a^d$ in $F$.\end{thm}

Remarkably, this theorem combined with Theorem~\ref{part2} and Proposition \ref{invSbgrp} implies that irreducible nonsurjective endomorphisms have word-hyperbolic mapping tori.

\begin{thm}\label{hypthm} If $\phi:F \to F$ is represented by a clean immersion, then $F*_\phi$ is word-hyperbolic. In particular, if $\phi$ is nonsurjective and irreducible, then $F*_\phi$ is word-hyperbolic. \end{thm}
\begin{proof} As $\phi$ is represented by a clean immersion, every nontrivial conjugacy class is $\phi$-expanding. Suppose $F*_\phi$ were not word-hyperbolic. By Theorem~\ref{mythm}, there exists a nontrivial element $a \in F$, an element $g \in F$, and integers $d, n \ge 1$ such that $\phi^n(a) = g a^d g^{-1}$. If we let $H = \langle a \rangle$, then $i_g \circ \phi^n(H) \le H$. By Proposition~\ref{invSbgrp}, there is a $k \ge 0$ such that $ (i_g\circ\phi^n)^k(F) \cap H$ has finite index in $(i_g \circ \phi^n)^k(F)$. But this is a contradiction as $H$ cannot be cyclic and have finite index intersection in a noncyclic free group. Therefore, $F*_\phi$ must be word-hyperbolic. The second statement of the theorem follows from Theorem~\ref{part2}.\end{proof}

On the other hand, there are fully irreducible automorphisms whose corresponding free-by-cylic groups are not word-hyperbolic. In this case, Bestvina-Handel showed that the automorphisms are induced by pseudo-Anosov homeomorphisms on once-punctured surfaces \cite[Proposition~4.5]{BH92}. By Thurston's hyperbolization theorem, the corresponding free-by-cyclic groups are fundamental groups of hyperbolic $3$-manifolds that fiber over a circle.

\bibliography{refs}

\begin{thebibliography}{10}

\bibitem{BR97}
R.~B. Bapat and T.~E.~S. Raghavan.
\newblock {\em Nonnegative matrices and applications}, volume~64 of {\em
  Encyclopedia of Mathematics and its Applications}.
\newblock Cambridge University Press, Cambridge, 1997.

\bibitem{BFH97e}
M.~Bestvina, M.~Feighn, and M.~Handel.
\newblock Erratum to: ``{L}aminations, trees, and irreducible automorphisms of
  free groups'' [{G}eom. {F}unct. {A}nal. 7(2):215--244, 1997].
\newblock {\em Geom. Funct. Anal.}, 7(6):1143, 1997.

\bibitem{BFH97}
Mladen Bestvina, Mark Feighn, and Michael Handel.
\newblock Laminations, trees, and irreducible automorphisms of free groups.
\newblock {\em Geom. Funct. Anal.}, 7(2):215--244, 1997.

\bibitem{BH92}
Mladen Bestvina and Michael Handel.
\newblock Train tracks and automorphisms of free groups.
\newblock {\em Ann. of Math. (2)}, 135(1):1--51, 1992.

\bibitem{BNS87}
Robert Bieri, Walter~D. Neumann, and Ralph Strebel.
\newblock A geometric invariant of discrete groups.
\newblock {\em Invent. Math.}, 90(3):451--477, 1987.

\bibitem{Bri00}
Peter Brinkmann.
\newblock Hyperbolic automorphisms of free groups.
\newblock {\em Geom. Funct. Anal.}, 10(5):1071--1089, 2000.

\bibitem{CV86}
Marc Culler and Karen Vogtmann.
\newblock Moduli of graphs and automorphisms of free groups.
\newblock {\em Invent. Math.}, 84(1):91--119, 1986.

\bibitem{DV96}
Warren Dicks and Enric Ventura.
\newblock {\em The group fixed by a family of injective endomorphisms of a free
  group}, volume 195 of {\em Contemporary Mathematics}.
\newblock American Mathematical Society, Providence, RI, 1996.

\bibitem{DKL15}
Spencer Dowdall, Ilya Kapovich, and Christopher~J. Leininger.
\newblock Dynamics on free-by-cyclic groups.
\newblock {\em Geom. Topol.}, 19(5):2801--2899, 2015.

\bibitem{DKL17}
Spencer Dowdall, Ilya Kapovich, and Christopher~J. Leininger.
\newblock Endomorphisms, train track maps, and fully irreducible monodromies.
\newblock {\em Groups Geom. Dyn.}, 11(4):1179--1200, 2017.

\bibitem{Kap14}
Ilya Kapovich.
\newblock Algorithmic detectability of iwip automorphisms.
\newblock {\em Bull. Lond. Math. Soc.}, 46(2):279--290, 2014.

\bibitem{Kap19}
Ilya Kapovich.
\newblock Detecting fully irreducible automorphisms: a polynomial time
  algorithm.
\newblock {\em Exp. Math.}, 28(1):24--38, 2019.
\newblock With an appendix by Mark C. Bell.

\bibitem{LL03}
Gilbert Levitt and Martin Lustig.
\newblock Irreducible automorphisms of {$F_n$} have north-south dynamics on
  compactified outer space.
\newblock {\em J. Inst. Math. Jussieu}, 2(1):59--72, 2003.

\bibitem{JPM2}
Jean~Pierre {Mutanguha}.
\newblock {Irreducibility of a Free Group Endomorphism is a Mapping Torus
  Invariant}.
\newblock {\em arXiv e-prints}, Oct 2019.
\newblock \href{https://arxiv.org/abs/1910.04285}{arXiv:1910.04285}.

\bibitem{JPM}
Jean~Pierre {Mutanguha}.
\newblock {Hyperbolic Immersions of Free Groups}.
\newblock {\em Groups Geom. Dyn.}, to appear.
\newblock \href{https://arxiv.org/abs/1809.04761}{arxiv:1809.04761}.

\bibitem{Pfaff}
Catherine Pfaff.
\newblock Ideal {W}hitehead graphs in {$Out(F_r)$} {II}: the complete graph in
  each rank.
\newblock {\em J. Homotopy Relat. Struct.}, 10(2):275--301, 2015.

\bibitem{Rey}
Patrick {Reynolds}.
\newblock {Dynamics of Irreducible Endomorphisms of $F_n$}.
\newblock {\em ArXiv e-prints}, August 2010.
\newblock \href{https://arxiv.org/abs/1008.3659}{arxiv:1008.3659}.

\bibitem{St83}
John~R. Stallings.
\newblock Topology of finite graphs.
\newblock {\em Invent. Math.}, 71(3):551--565, 1983.

\bibitem{Thu82}
William~P. Thurston.
\newblock Three-dimensional manifolds, {K}leinian groups and hyperbolic
  geometry.
\newblock {\em Bull. Amer. Math. Soc. (N.S.)}, 6(3):357--381, 1982.

\bibitem{Thu88}
William~P. Thurston.
\newblock On the geometry and dynamics of diffeomorphisms of surfaces.
\newblock {\em Bull. Amer. Math. Soc. (N.S.)}, 19(2):417--431, 1988.

\bibitem{Vogt02}
Karen Vogtmann.
\newblock Automorphisms of free groups and outer space.
\newblock In {\em Proceedings of the {C}onference on {G}eometric and
  {C}ombinatorial {G}roup {T}heory, {P}art {I} ({H}aifa, 2000)}, volume~94,
  pages 1--31, 2002.

\bibitem{Vogt17}
Karen Vogtmann.
\newblock Contractibility of outer space: reprise.
\newblock In {\em Hyperbolic geometry and geometric group theory}, volume~73 of
  {\em Adv. Stud. Pure Math.}, pages 265--280. Math. Soc. Japan, Tokyo, 2017.

\end{thebibliography}
\bibliographystyle{plain}
\end{document}